\documentclass[12pt,a4paper]{amsart}

\usepackage{tikz}
\usepackage{pgffor}
\usepackage[centering]{geometry}
\usepackage{hyperref}

\newcommand{\field}[1]{\ensuremath{\mathbb{#1}}}
\newcommand{\integer}{\field{Z}}

\newcommand{\col}[1]{\ensuremath\mathcal{C}_{#1}}
\newcommand{\hcol}[1][1]{\ensuremath\hat{\mathcal{C}}_{#1}}

\newtheorem{theorem}{Theorem}
\newtheorem{lemma}{Lemma}[section]
\newtheorem{corollary}[lemma]{Corollary}
\numberwithin{equation}{section}

\theoremstyle{definition}
\newtheorem*{definition}{Definition}
\newtheorem*{mainproblem}{Main Problem}
\newtheorem*{mainproblemdash}{Main Problem, reformulated}

\theoremstyle{remark}
\newtheorem{remark}[lemma]{Remark}

\begin{document}

\title[Maximising Common Fixtures in a Round Robin with Two Divisions]{Maximising Common Fixtures in a Round Robin Tournament with Two Divisions}

\author{Wayne Burrows}
\address{Institute of Fundamental Sciences, Massey University, Private Bag 11222, Palmerston North 4442, New Zealand}
\email{w.j.burrows@massey.ac.nz}

\author{Christopher Tuffley}
\address{Institute of Fundamental Sciences, Massey University, Private Bag 11222, Palmerston North 4442, New Zealand}
\email{c.tuffley@massey.ac.nz}

\date{\today}

\subjclass[2010]{05C70}
\keywords{Sports scheduling, tournament, round robin draw, common fixture, edge colouring, one-factorisation, bipyramidal}

\begin{abstract}
We describe a round robin scheduling problem for a competition played in two divisions, motivated by a scheduling problem brought to the second author by a local sports organisation. The first division has teams from $2n$ clubs, and is played in a double round robin in which the draw for the second round robin is identical to the first. The second division has teams from two additional clubs, and is played as a single round robin during the first $2n+1$ rounds of the first division. We will say that two clubs have a \emph{common fixture} if their teams in division one and two are scheduled to play each other in the same round, and show that for $n\geq2$ the maximum possible number of common fixtures is $2n^2 - 3n + 4$. Our construction of draws achieving this maximum is based on a bipyramidal one-factorisation of $K_{2n}$, which represents the draw in division one. Moreover, if we additionally require the home and away status of common fixtures to be the same in both divisions, we show that the draws can be chosen to be balanced in all three round robins. 
\end{abstract}

\maketitle

\section{Introduction}

We discuss a round robin tournament scheduling problem played in two divisions, with the objective to maximise the number of common fixtures between two clubs playing against each other in the same round in the two separate divisions. The first division has teams from $2n$ clubs, and is played in a double round robin in which the draw for the second round robin is identical to the first. The second division has teams from two additional clubs, and is played as a single round robin during the first $2n+1$ rounds of the first division. We say that two clubs have a \emph{common fixture} if their division one and two teams both play each other in the same round, and show that for $n\geq2$ the maximum possible number of common fixtures is $2n^2 - 3n + 4$. Our construction achieving this bound is based on a bipyramidal one-factorisation of the complete graph $K_{2n}$.

This problem was motivated by a scheduling problem in the Manawat\=u Rugby Union's first and second division tournaments in New Zealand in 2011. In that case there were ten clubs with a team in both divisions, and an additional two clubs with teams in the second division only. The Manawat\=u Rugby Union contacted the second author to request help in designing a schedule to maximise the number of common fixtures. A near optimal schedule was found by the second author and implemented by the rugby union. We solve the problem for any number of clubs in the first division, with two additional clubs in the second division.

\subsection{Organisation}

The paper is organised as follows. 
In Section~\ref{sec:theproblem} we give a precise statement of our problem, reformulate it in graph-theoretic terms, and state our main theorem. In Section~\ref{sec:upperbound} we establish the upper bound given in our theorem, and in Section~\ref{sec:construction} we construct draws achieving this bound. This is done in two parts: we first handle the case $n=2$ separately in Section~\ref{sec:n=2}, and then give a general construction for $n\geq 3$ in Section~\ref{sec:general}. We conclude the paper in Section~\ref{sec:homeaway} by considering an oriented version of the problem, representing home and away status, and show that the draws can be chosen to be balanced.

\subsection{Related work}

In theory and application it is often desirable to construct a sports schedule subject to additional constraints or objectives. Many such problems have been investigated, such as:  
Victoria Golf Association scheduling two divisions to avoid clashes~\cite{bib:Beecham&Hurley}; scheduling $n$ teams from each of two geographic locations so that games between teams from the same location take place on weekdays, and games between teams from different locations take place at weekends~\cite{bib:dewerra1980};
shared facilities in an Australian Basketball Association~\cite{bib:deWerra&Jacot-Descombes&Masson}; 
scheduling a round robin tennis tournament under availability constraints on courts and players~\cite{bib:DellaCroce&Tadei&Asioli}; 
minimising ``carry-over effects'', where teams $x$ and $y$ both play team $z$ immediately after playing team $w$~\cite{bib:anderson1997}; 
avoiding consecutive away fixtures in the Czech National Basketball League~\cite{bib:Froncek}; 
minimising waiting times in tournaments played on a single court~\cite{bib:knust2008}; 
scheduling to avoid teams playing consecutive games against teams from the same ``strength group''~\cite{bib:briskorn2009,bib:briskorn-kunst2010}; 
minimising breaks (consecutive home or away games)~\cite{bib:Hof}; a travelling tournament problem, where it is desirable to have a number of consecutive away games (on tour) applied to a Japanese baseball league~\cite{bib:Hoshino}. 
See Wallis~\cite[Chapter 5]{bib:Wallis} or Kendall et al's comprehensive survey article~\cite{bib:Kendall} and the references therein for further discussion and examples. 

In problems involving teams that share facilities (for example, teams belonging to the same club but playing in different divisions, as we consider here) it is common to apply the constraint that such teams cannot have a home game in the same round (see for example~\cite{bib:Beecham&Hurley,bib:deWerra&Jacot-Descombes&Masson} and~\cite[p. 35]{bib:Wallis}). This reflects the common situation where it may be physically impossible to conduct two games at the same time at the same venue. In this paper we drop this constraint, and instead seek to maximise the number of games between teams from the same two clubs, played in the same round and at the same venue. This might for example allow the club's teams to share transport, reducing the costs associated with travel. The scheduling difficulty in the problem considered here arises from the fact that not all clubs have a team in both divisions.

\section{Problem statement}
\label{sec:theproblem}

\subsection{Setting}

Our interest in this paper is in \emph{round robin tournaments}: tournaments in which every team or competitor taking part in the competition plays against every other team or competitor exactly once (a [\emph{single}] \emph{round robin}) or twice (a \emph{double round robin}). For simplicity we will use the term \emph{team} throughout (that is, we allow teams consisting of one player only), since the number of players in a team plays no role in our discussion. We assume that the round robin tournament takes place as a series of \emph{rounds}, in which each team plays exactly one match against another team. To handle the case where there is an odd number of teams we follow common practice by introducing a phantom team; when a team is scheduled to play against the phantom team they have a bye in that round. Thus in what follows we will always assume that there is an even number of teams.

We will regard the teams as belonging to \emph{clubs}, and will assume that each club may enter at most one team in a given tournament. However, a club may have more than one team (for example, an ``A'' and a ``B'' team, a junior and a senior team, or a men's and a women's team) that take part in different tournaments. We will refer to each tournament and associated set of participating teams as a \emph{division}. 

In some sports or tournaments one of the two teams taking part in a given match may be in a distinguished position. This is the case for example where one team plays first, or where matches take place at the facilities belonging to one of the two teams, with the team playing at their own facilities being the ``home'' team, and the team travelling to the other's facilities being the ``away'' team. For simplicity we will use the terms \emph{home} and \emph{away} throughout to specify this distinction. 

The \emph{draw} for a tournament specifies which matches take place in each round. In some sports home and away are decided by lot, whereas in others it must be specified as part of the draw. In such cases it is desirable that every team has nearly equal numbers of home games, and we will say that a draw is \emph{balanced} if the numbers of home games of any two teams differ by at most one. (Note that in a single round robin with $2n$ teams it is impossible for all teams to have the same number of home games, because each team plays $2n-1$ games.) We will use the term \emph{fixture} to refer to a match scheduled to take place in a particular round, with if applicable a designation of home and away teams.

\subsection{Formulation}

Let $n$ be a positive integer. 
We consider a competition played in two divisions among $2n+2$ clubs labelled $0,1, \dots, 2n+1$. We suppose that
\begin{enumerate}
\renewcommand{\theenumi}{(C\arabic{enumi})}
\renewcommand{\labelenumi}{(C\arabic{enumi})}
\item
clubs $0, 1, \ldots, 2n-1$ have a team competing in each division;
\item
clubs $2n$ and $2n+1$ have teams competing in division two only;
\item\label{double.condition}
division one is played as a double round robin, in which the draws for rounds $r$ and $r+(2n-1)$ are identical for $r=1,\ldots,2n-1$, but with (if applicable) home and away reversed;
\item\label{coincide.condition}
division two is played as a single round robin, co-inciding with the first $2n+1$ rounds of division one.
\end{enumerate}
We will say that clubs $x$ and $y$ have a \textbf{common fixture} in round $r$ if their division one and two teams both play each other in round $r$. When home and away are specified as part of the draw we additionally require that the same club should be the home team in both divisions.

It is clear that there are circumstances in which common fixtures might be desirable. For example, they might allow a club's division one and two teams to share transport, and they might allow the club's supporters to attend both the division one and two games. This motivates our main problem:

\begin{mainproblem}
Construct round robin draws maximising the total number of common fixtures among clubs $0,1,\ldots,2n-1$.
\end{mainproblem}
Our construction yields the following result:

\begin{theorem}\label{thm:main}
Let $n$ be a positive integer. Then the maximum possible number of common fixtures is $1$ if $n=1$, and $c(n)=2n^2-3n+4$ if $n\geq 2$. Moreover, if home and away are specified the draws can be chosen to be balanced in all three round robins (division two, and both round robins of division one).
\end{theorem}

When $n=1$ there are only two teams in division one, and four teams in division two. It is immediate that there can be at most one common fixture, and any draw for division two in which the clubs belonging to division one play each other in round one or two realises this. We therefore restrict attention to  $n\geq2$ throughout the rest of the paper.

\begin{remark}\label{rem:singleroundrobin}
We have assumed that division one is played as a double round robin, and division two as a single round robin, because that is the form in which the problem was presented to us by a sports organisation in 2011. 
If division one is played only as a single round robin, then our work shows that for $n\geq2$ the maximum possible number of common fixtures is $c(n)-2=2n^2-3n+2$.
\end{remark}

\begin{remark}
Our sequence $c(n)$ is a translate by 1 of sequence 
\href{http://oeis.org/A236257}{A236257} in the Online Encyclopedia of Integer Sequences (OEIS), published electronically at \href{http://oeis.org}{\nolinkurl{http://oeis.org}}. This sequence is defined by $a(n) = 2n^2 - 7n + 9=c(n-1)$, and relates to sums of $n$-gonal numbers.

The sequence $c(n)-2$ is a translate by 1 of sequence
\href{http://oeis.org/A084849}{A084849} in the OEIS, defined by $a(n) = 2n^2+n+1=c(n+1)-2$. This sequence counts the number of ways to place two non-attacking bishops on a $2\times (n+1)$ board.
\end{remark}

\subsection{Reformulation in graph-theoretic terms}

To formulate the problem in graph-theoretic terms we follow standard practice and represent each team by a vertex, and a match between teams $x$ and $y$ by the edge $\{x,y\}$. In a round robin tournament with $2m$ teams each round then corresponds to a perfect matching or \emph{one-factor} of the complete graph $K_{2m}$, and the round robin draw to an ordered one-factorisation of $K_{2m}$ (Gelling and Odeh~\cite{bib:Gelling&Odeh}, de Werra~\cite{bib:dewerra1980}). Recall that these terms are defined as follows.

\begin{definition}
A \emph{one-factor} or \emph{perfect matching} of $G=(V,E)$ is a subgraph $\bar{G}=(V,\bar{E})$ of $G$ in which the edges $\bar{E} \subseteq E$ have the following properties:
\begin{enumerate}
\item Every vertex $v\in V$ is incident on an edge $e \in \bar{E}$.
\item No two edges $e$ and $e'$ in $\bar{E}$ have any vertex in common.
\end{enumerate}
As a consequence every vertex $v \in V$ has degree one in $\bar{G}$.
A \emph{one-factorisation} of $G$ is a set of one-factors $\{\bar{G}_i=(V,\bar{E}_i) \mid i=1,\ldots,k\}$ with the properties:
\begin{enumerate}
\item 
$\bar{E}_i \cap \bar{E}_j = \emptyset, \; i \ne j$.
\item 
$\bigcup\limits_{i=1}^k \bar{E}_i = E$.
\end{enumerate}
Clearly, a necessary condition for $G$ to have a one-factorisation into $k$ one-factors is that $G$ is regular of degree $k$. In particular, for $G=K_{2m}$ any one-factorisation must have $k=2m-1$ one-factors. 
\end{definition}
Any one-factorisation can be thought of as an edge colouring of the 
given graph, and in the case of the complete graph $K_{2m}$, a one-factorisation is equivalent to a minimum edge colouring.
In what follows we will be interested in the cases $m=n$ and $m=n+1$.
We will use the languages of one-factorisations and edge colourings interchangeably.
Note that a one-factorisation or minimum edge colouring does not necessarily impose an order on the one-factors. If an order is fixed, we will say the one-factorisation is ordered or we have an ordered one-factorisation.

Turning now to the problem, suppose that the complete graph $K_{2n}=(\mathcal{V}_1,\mathcal{E}_1)$ has vertex set $\mathcal{V}_1=\{0,1,\ldots,2n-1\}$, and 
$K_{2n+2}=(\mathcal{V}_2,\mathcal{E}_2)$ has vertex set $\mathcal{V}_2=\{0,1,\ldots,2n+1\}$. Then the round robin draw in division one rounds 1 to $2n-1$ may be represented by an edge colouring
\[
\col{1} : \mathcal{E}_1 \to \{1, 2,  \ldots, {2n-1}\},
\]
where the edges coloured $r$ represent the draw in round $r$. By condition~\ref{double.condition} the draw in rounds $2n$ to $4n-2$ is then given by the colouring 
\[
\hcol : {\mathcal{E}_1} \to \{2n, 2n+1, \ldots,{4n-2}\}
\]
defined by $\hcol(e)=\col{1}(e)+(2n-1)$, and by condition~\ref{coincide.condition} the draw in division 2 may be represented by a colouring
\[
\col{2} : \mathcal{E}_2 \to \{1, 2, \ldots ,{2n+1}\}. 
\]
Clubs $x$ and $y$ therefore have a common fixture in round $r$ if and only if
\[
\col{2}(\{x,y\}) = r \in \{\col{1}(\{x,y\}),\hcol(\{x,y\})\};
\]
since $\hcol(\{x,y\})=\col{1}(\{x,y\})+(2n-1)$ this may be expressed concisely as
\[
\col{2}(\{x,y\}) = r \equiv \col{1}(\{x,y\}) \bmod (2n-1).
\]
Our problem may then be stated as follows:
\begin{mainproblemdash}
Let
$K_{2n}=(\mathcal{V}_1,\mathcal{E}_1)$ have vertex set $\mathcal{V}_1=\{0,1,\ldots,2n-1\}$, and let
$K_{2n+2}=(\mathcal{V}_2,\mathcal{E}_2)$ have vertex set $\mathcal{V}_2=\{0,1,\ldots,2n+1\}$.
Construct proper edge colourings 
\begin{align*}
\col{1} &: \mathcal{E}_1 \to \{1, 2,  \ldots, {2n-1}\}, \\
\col{2} &: \mathcal{E}_2 \to \{1, 2, \ldots ,{2n+1}\},
\end{align*}
of $K_{2n}$ and $K_{2n+2}$, respectively, maximising the number of edges $\{x,y\}\in \mathcal{E}_1$ such that
\begin{equation}\label{common-double.eq}
\col{2}(\{x,y\}) \equiv \col{1}(\{x,y\}) \bmod (2n-1).
\end{equation}
\end{mainproblemdash}

\begin{remark}
When division one is played as a single round robin then the condition of equation~\eqref{common-double.eq} for clubs $x$ and $y$ to have a common fixture becomes simply
\[
\col{2}(\{x,y\}) = \col{1}(\{x,y\}).
\]
\end{remark}

\begin{remark}
When applicable we will orient the edges to indicate the home and away status of a game, with the edges pointing from the home team to the away team. In that case we additionally require that identically coloured edges have the same orientation. We address home and away status in Section~\ref{sec:homeaway}.
\end{remark}

\section{The upper bound}
\label{sec:upperbound}

In this section we show that $c(n)=2n^2 - 3n + 4$ is an upper bound on the number of common fixtures. Recall that we assume $n\geq2$ throughout.

Division one involves $2n$ teams, so in each round there are exactly $n$ games. Thus in each round there can be at most $n$ common fixtures. However, in Lemma~\ref{lem:specialround} we show that there is at most one round in which this can occur. We then show in Lemma~\ref{lem:RR2} that condition~\ref{double.condition} constrains the total number of common fixtures that can occur in rounds 1 and $2n$ to at most $n$, and similarly in rounds $2$ and $2n+1$. Combining these conditions gives $c(n)$ as an upper bound.

\begin{lemma}
\label{lem:specialround}
There is at most one round in which there are $n$ common fixtures. For every other round there are at most $n-1$ common fixtures.
\end{lemma}

\begin{proof}
In each round of division two there are $n+1$ games. In exactly one round the teams from the additional clubs $2n$ and $2n+1$ play each other, leaving  $(n+1) - 1 = n$ games between the $2n$ clubs common to both divisions in which it is possible to have a common fixture.

In every other round the clubs $2n$ and $2n+1$ each play a club that is common to both divisions. This leaves $(n+1)-2 = n-1$ games between clubs common to both divisions in which it is possible to have a common fixture.
\end{proof}

Recall by condition~\ref{coincide.condition} that the draws for the first and second round robins in division one are identical. This constrains the total number of common fixtures between the pairs of identical rounds in the two round robins of the first division.

\begin{lemma}
\label{lem:RR2}
In total there are at most $n$ common fixtures in rounds $1$ and $2n$. Similarly, in total there are at most $n$ common fixtures in rounds $2$ and $2n+1$.
\end{lemma}

\begin{proof} 
Rounds $1$ and $2n$ correspond to the first round of the first round robin in division one, and the first round of the second round robin in division one. Since the fixtures in these rounds are identical (disregarding the home and away status), and each fixture occurs once only in division two, there are at most $n$ distinct fixtures and therefore at most $n$ common fixtures in total between the two rounds.

By an identical argument, rounds 2 and $2n+1$ have in total at most $n$ common fixtures also.
\end{proof}

\begin{corollary}
\label{cor:bound}
The number of common fixtures is at most $c(n)=2n^2-3n+4$. For this to be possible the game between teams $2n$ and $2n+1$ in division two must take place in one of rounds $3$ to $2n-1$. 
\end{corollary}

\begin{proof} 
Let $f_r$ be the number of common fixtures in round $r$, $1\le r \le 2n+1$. We want to bound the total number of common fixtures, which is $\sum_{r=1}^{2n+1} f_r$.

Suppose that the game between teams $2n$ and $2n+1$ occurs in round $q$. Then, by Lemmas~\ref{lem:specialround} and~\ref{lem:RR2} we have: 
\begin{align*}
f_{q} & \le n, \\
f_r & \le n-1, & r &\ne q, \\
f_r + f_{2n-1+r} &  \le n, & r &\in \{ 1, 2 \}.
\end{align*}
If $q \in \{ 1, 2, 2n, 2n+1\}$ 
then \begin{align*}
\sum_{r=1}^{2n+1} f_r & = f_1 + f_2 + f_{2n} + f_{2n+1} + \sum_{r=3}^{2n-1} f_r \\
& \le 2n + (2n-3)(n-1) \\
& = 2n^2 - 3n + 3.
\end{align*}
Otherwise, we have $q \in \{ 3, 4, \ldots, 2n-1\}$ and
\begin{align*}
\sum_{r=1}^{2n+1} f_r & = f_1 + f_2 + f_{2n} + f_{2n+1} + \sum_{r=3}^{2n-1} f_r \\
& \le 2n + (2n-4)(n-1) + n\\
& = 2n^2 - 3n + 4.
\end{align*}
In either case we have $\sum_{r=1}^{2n+1} f_r \le 2n^2 - 3n + 4$, with equality possible only when $q \in \{ 3,4, \ldots , 2n-1\}. $ 
\end{proof}

\begin{remark}\label{rem:lowerbound}
When division one is played as a single round robin, the above argument shows that the number of common fixtures is at most
\[
n+(2n-2)(n-1) = 2n^2 -3n +2 = c(n)-2.
\]
\end{remark}

\section{The construction}
\label{sec:construction}

In this section we construct one-factorisations 
$\mathcal{F}^1 = \{F_r^1 \mid 1\leq r\leq 2n-1\}$ of $K_{2n}$ and $\mathcal{F}^2 = \{F_r^2 \mid 1\leq r\leq 2n+1\}$ of $K_{2n+2}$ realising the upper bound $c(n)$ of Corollary~\ref{cor:bound}. Here each one-factor $F_r^d$ represents the draw in round $r$ of division $d$. In the general case $n\geq 3$ our construction
uses a \emph{factor-1-rotational}~\cite{bib:Mendelsohn&Rosa} one-factorisation of $K_{2n}$, also known as a  
\emph{bipyramidal}~\cite{bib:Mazzuoccolo&Rinaldi} one-factorisation. This construction does not apply when $n=2$, so we first handle this case separately in Section~\ref{sec:n=2}, before giving our general construction in Section~\ref{sec:general}. We conclude by discussing home and away status for $n\geq 3$ in Section~\ref{sec:homeaway}.

\subsection{The case $n=2$}
\label{sec:n=2}

When $n=2$ we define the required one-factorisations
$\mathcal{F}^1 = \{F_r^1 \mid 1\leq r\leq 3\}$ of $K_{4}$ and 
$\mathcal{F}^2 = \{F_r^2 \mid 1\leq r\leq 5\}$ of $K_{6}$ as follows:
\begin{align*}
F_1^1 &=\{(0,1),\mathbf{(2,3)}\}, & F_1^2 &= \{\mathbf{(2,3)},(4,0),(5,1)\}, \\
F_2^1 &=\{\mathbf{(2,0)},(3,1)\}, & F_2^2 &= \{\mathbf{(2,0)},(3,5),(4,1)\}, \\
F_3^1 &=\{\mathbf{(0,3)},\mathbf{(1,2)}\}, 
      & F_3^2 &= \{\mathbf{(0,3)},\mathbf{(1,2)},(4,5)\}, \\
&&  F_4^2 &= \{\mathbf{(1,0)},(3,4),(5,2)\}, \\
&&  F_5^2 &= \{(0,5),\mathbf{(1,3)},(2,4)\}.
\end{align*}
The draw is also shown graphically in Figure~\ref{fig:n=2}. The common fixtures are indicated in bold, and we see that there are a total of $c(2)=2\cdot2^2-3\cdot2+4=6$ of them. Moreover, with the edges oriented as given we see that pairs of edges corresponding to common fixtures are identically oriented, and that every vertex in division one has outdegree either 1 or 2, and every vertex in division two has outdegree either 2 or 3. Thus, all three round robin draws are balanced, and together with Corollary~\ref{cor:bound} this establishes Theorem~\ref{thm:main} in the case $n=2$.

\begin{figure}
\begin{tikzpicture}[vertex/.style={circle,draw,fill=black!20},>=stealth]
\node (1-0) at (90:2) [vertex] {0};
\node (1-1) at (150:2) [vertex] {1};
\node (1-2) at (210:2) [vertex] {2};
\node (1-3) at (270:2) [vertex] {3};
\node at (0,-3) {(a)};
\begin{scope}[xshift = 100]
\node (1_0) at (90:2) [vertex] {0};
\node (1_1) at (150:2) [vertex] {1};
\node (1_2) at (210:2) [vertex] {2};
\node (1_3) at (270:2) [vertex] {3};
\node at (0,-3) {(b)};
\end{scope}
\begin{scope}[xshift = 200]
\node (2-0) at (90:2) [vertex] {0};
\node (2-1) at (150:2) [vertex] {1};
\node (2-2) at (210:2) [vertex] {2};
\node (2-3) at (270:2) [vertex] {3};
\node (2-4) at (330:2) [vertex] {4};
\node (2-5) at (30:2) [vertex] {5};
\node at (0,-3) {(c)};
\end{scope}
\begin{scope}[red,->,thick]
\draw (1-0) to (1-1);
\draw [ultra thick] (1-2) to (1-3);
\draw [ultra thick] (2-2) to (2-3);
\draw (2-4) to (2-0);
\draw (2-5) to (2-1);
\end{scope}
\begin{scope}[blue,->,thick,dashed]
\draw [ultra thick] (1-2) to (1-0);
\draw (1-3) to (1-1);
\draw [ultra thick] (2-2) to (2-0);
\draw (2-3) to (2-5);
\draw (2-4) to (2-1);
\end{scope}
\begin{scope}[green,->,thick,densely dotted]
\draw [ultra thick] (1-0) to (1-3);
\draw [ultra thick] (1-1) to (1-2);
\draw [ultra thick] (2-0) to (2-3);
\draw [ultra thick] (2-1) to (2-2);
\draw (2-4) to (2-5);
\end{scope}
\begin{scope}[magenta,->,thick,dash pattern = on 4pt off 1pt on 1pt off 1pt]
\draw [ultra thick] (1_1) to (1_0);
\draw (1_3) to (1_2);
\draw [ultra thick] (2-1) to (2-0);
\draw (2-3) to (2-4);
\draw (2-5) to (2-2);
\end{scope}
\begin{scope}[cyan,->,thick,dash pattern = on 4pt off 1pt on 1pt off 1pt on 1pt off 1pt]
\draw (1_0) to (1_2);
\draw [ultra thick] (1_1) to (1_3);
\draw [ultra thick] (2-1) to (2-3);
\draw (2-2) to (2-4);
\draw (2-0) to (2-5);
\end{scope}
\begin{scope}[black,<-,thick,dash pattern = on 10pt off 2 pt ]
\draw (1_0) to (1_3);
\draw (1_1) to (1_2);
\end{scope}
\end{tikzpicture}
\caption{The draws for $n=2$, with common fixtures denoted by thicker edges. (a) The draw in division one rounds 1--3, with the rounds denoted by red solid edges; blue dashed edges; and green dotted edges, respectively. (b) The draw in division one rounds 4--6, with the rounds denoted by magenta dash-dotted edges; cyan dash-dot-dotted edges; and black dashed edges, respectively. (c) The draw in division two, with the rounds denoted as above.}
\label{fig:n=2}
\end{figure}
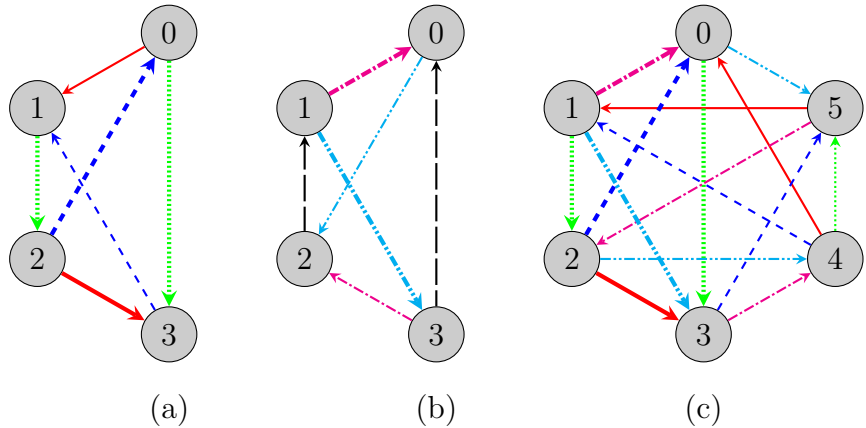

\subsection{The general case $n\geq 3$}
\label{sec:general}

\subsubsection{Overview}
\label{sec:overview}

In the general case $n\geq 3$ our draw in division one is based on a class of  one-factorisations of $K_{2n}$ known as \emph{factor-1-rotational}~\cite{bib:Mendelsohn&Rosa} or \emph{bipyramidal}~\cite{bib:Mazzuoccolo&Rinaldi}. Such a one-factorisation is obtained by first constructing a single one-factor, known as a \emph{starter}. Two of the vertices are then held fixed, while the remaining vertices are permuted according to the sharply transitive action of a group $G$ of order $2n-2$. In our case we use the cyclic group of order $2n-2$. This produces $2n-2$ one-factors, and by careful choice of the initial one-factor and group action these are all disjoint, and are completed to a one-factorisation by the addition of a final one-factor, that is fixed by the action of $G$ and consists of the remaining edges.

In order to achieve the close agreement required between the division one and two draws we exploit the symmetry of the division one draw in constructing the draw for division two. 
We begin by modifying the starter one-factor of $K_{2n}$, by replacing one of its edges with a pair of edges joining its endpoints to the two additional vertices. This gives us $n-1$ common fixtures in round 1. We then translate this one-factor by the action of $G$, to obtain $n-1$ common fixtures in each of rounds 2 to $2n-2$ as well. The draw for division one round $2n-1$ is described by the fixed one-factor of $K_{2n}$, and adding the edge between the two additional teams to this gives us a round in which there are $n$ common fixtures. It then remains to organise the remaining 
edges --- those removed from the cyclicly permuted one-factors, as well as the remaining edges between the fixed vertices --- into two more rounds, in such a way that we pick up an additional common fixture in each. We will ensure that this is possible by choosing the edge removed from the starter so that its orbit forms a cycle in the graph of length $2n-2$, and so consequently has a one-factorisation.

\subsubsection{The construction}

In order to describe the construction, it will be convenient to denote the vertices $2n-2$ and $2n-1$ by $\pm\infty$, and the vertices $2n$ and $2n+1$ by $\pm i\infty$. Then we may unambiguously define the permutation $\sigma$  of
$\mathcal{V}_2=\{0,\ldots,2n-3\}\cup\{\pm\infty\}\cup\{\pm i\infty\}$ by $\sigma(x)=x+1$, where addition is done modulo $2n-2$ for $0\leq x\leq 2n-3$, and $x+k=x$ for $x\in\{\pm\infty,\pm i\infty\}$, $k\in\integer$. The group $G=\langle\sigma\rangle$ is cyclic of order $2n-2$, and acts sharply transitively on the vertices $\{0,1,\ldots,2n-3\}$.

We begin by constructing the one-factors $F_1^1$ and $F_1^2$ in Lemma~\ref{lem:round1}. The cases $n=7$ and $n=8$ are illustrated in Figures~\ref{fig:F1 for n=7} and~\ref{fig:F1 for n=8}, respectively.


\begin{lemma} 
\label{lem:round1}
Define
\begin{align*}
s &= \begin{cases} 
     \frac{n-4}{2}  & \text{$n$ even,}\\ 
     \frac{n-3}{2}  & \text{$n$ odd,}
     \end{cases} & 
t &= n-2-s = 
     \begin{cases} 
     s+2=\frac{n}{2}    & \text{$n$ even,}\\ 
     s+1=\frac{n-1}{2}  & \text{$n$ odd,}
     \end{cases} \\
u &= \begin{cases} 
     \frac{3n-6}{2} & \text{$n$ even,}\\ 
     \frac{3n-7}{2} & \text{$n$ odd,}
     \end{cases} & 
v &= 3n-5-u = 
     \begin{cases} 
     u+1=\frac{3n-4}{2} & \text{$n$ even,}\\ 
     u+2=\frac{3n-3}{2} & \text{$n$ odd,}
     \end{cases}
\end{align*}
and let
\begin{align*}
E_1 &= \bigl\{\{x,y\}:x+y=n-2,0\leq x\leq s\bigr\} \\
    &= \bigl\{ \{0,n-2 \}, \{1,n-3\}, \ldots ,\{ s,t \}\bigr\}, \\ 
E_2 &= \bigl\{\{x,y\}:x+y=3n-5,n-1\leq x\leq u\bigr\} \\
    &= \bigl\{ \{n-1,2n-4 \}, \{n,2n-5\}, \ldots, \{ u,v \}\bigr\}, \\
E_3 &= \begin{cases}
  \bigl\{\{\frac{n-2}{2},-\infty\},\{2n-3,\infty\}\bigr\}, &\text{$n$ even,} \\
  \bigl\{\{\frac{3n-5}{2},-\infty\},\{2n-3,\infty\}\bigr\}, &\text{$n$ odd,} 
\end{cases} \\
E_4 &= \begin{cases}
  \bigl\{\{u,-i\infty\},\{v,i\infty\}\bigr\}, &\text{$n$ even,} \\
  \bigl\{\{s,-i\infty\},\{t,i\infty\}\bigr\}, &\text{$n$ odd.} 
\end{cases}
\end{align*}
Then 
\[
F_1^1=E_1 \cup E_2 \cup E_3
\]
is a one-factor of $K_{2n}$,
and 
\[
F_1^2 = \begin{cases}
        (F_1^1\cup E_4)-\bigl\{\{u,v\}\bigr\}, & \text{$n$ even,} \\
        (F_1^1\cup E_4)-\bigl\{\{s,t\}\bigr\}, & \text{$n$ odd}
        \end{cases}
\]
is a one-factor of $K_{2n+2}$. Moreover, $F_1^1$ and $F_1^2$ have precisely $n-1$ edges in common.
\end{lemma}

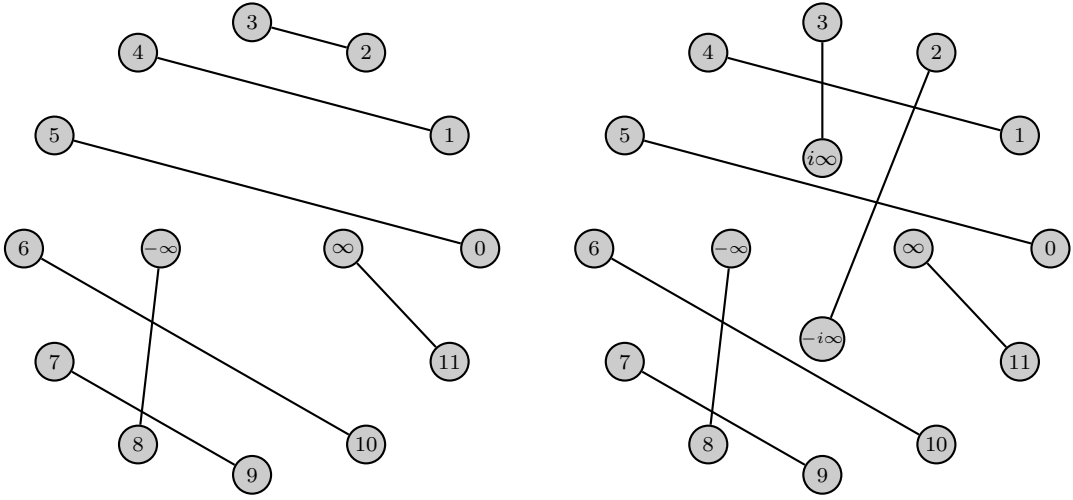
\begin{figure}
\begin{tikzpicture}[vertex/.style={circle,draw,fill=black!20,minimum size = 5mm,inner sep=0pt},>=stealth,thick]
\newlength{\radius}
\setlength{\radius}{3cm}
\foreach \x in {0,1,...,11}
   \node (\x-1) at (30*\x:\radius) [vertex] {$\scriptstyle\x$};
\node (-inf-1) at (180:0.4*\radius) [vertex] {$\scriptscriptstyle -\infty$};
\node (inf-1) at (0:0.4*\radius) [vertex] {$\scriptstyle\infty$};
\begin{scope}[xshift=2.5*\radius]
\foreach \x in {0,1,...,11}
   \node (\x-2) at (30*\x:\radius) [vertex] {$\scriptstyle\x$};
\node (-inf-2) at (180:0.4*\radius) [vertex] {$\scriptscriptstyle -\infty$};
\node (inf-2) at (0:0.4*\radius) [vertex] {$\scriptstyle\infty$};
\node (iinf) at (90:0.4*\radius) [vertex] {$\scriptstyle i\infty$};
\node (-iinf) at (270:0.4*\radius) [vertex] {$\scriptscriptstyle -i\infty$};
\end{scope}
\draw (0-1) -- (5-1);
\draw (1-1) -- (4-1);
\draw (2-1) -- (3-1);
\draw (6-1) -- (10-1);
\draw (7-1) -- (9-1);
\draw (8-1) -- (-inf-1);
\draw (11-1) -- (inf-1);
\draw (0-2) -- (5-2);
\draw (1-2) -- (4-2);
\draw (2-2) -- (-iinf);
\draw (3-2) -- (iinf);
\draw (6-2) -- (10-2);
\draw (7-2) -- (9-2);
\draw (8-2) -- (-inf-2);
\draw (11-2) -- (inf-2);
\end{tikzpicture}
\caption{The one-factors $F_1^1$ (left) and $F_1^2$ (right) in the case $n=7$.}
\label{fig:F1 for n=7}
\end{figure}


\begin{proof}
It is easy to check that each of the $2n$ vertices $x\in \mathcal{V}_1=\{0,1,\ldots,2n-3\}\cup\{\pm\infty\}$ belongs to precisely one edge in the union $E_1 \cup E_2 \cup E_3$. For $0\leq x\leq n-2$ the edge containing $x$ belongs to $E_1$, unless $n$ is even and $x=\frac{n-2}{2}$, in which case it belongs to $E_3$. For $n-1\leq x\leq 2n-4$ the edge belongs to $E_2$, unless $n$ is odd and $x=\frac{3n-5}2$, in which case it belongs to $E_3$; and for $x=2n-3$ and $x\in\{\pm\infty\}$ the edge belongs to $E_3$. 

When $n$ is even we obtain $F_1^2$ from $F_1^1$ by deleting the edge $\{u,v\}$ and adding the edges $\{u,-i\infty\}$ and $\{v,i\infty\}$; while when $n$ is odd we obtain $F_1^2$ from $F_1^1$ by deleting the edge $\{s,t\}$ and adding the edges $\{s,-i\infty\}$ and $\{t,i\infty\}$. Thus each vertex of $K_{2n+2}$ belongs to precisely one edge of $F_1^2$ also. Since $F_1^1$ contains precisely $n$ edges it follows moreover that $|F_1^1\cap F_1^2|=n-1$, as claimed.
\end{proof}

The one-factor $F_1^1$ is the starter discussed above in Section~\ref{sec:overview}.
From each one-factor $F_1^d$, $d=1,2$, we now construct $2n-2$ one-factors $F_r^d$,  $1 \le r \le 2n-2$, by permuting the vertices $\{ 0, 1, \dots 2n-3\}$ according to the permutation $\sigma = (0,1,\ldots,2n-3)$. For $1\leq r\leq 2n-2$ and $d=1,2$ we define
\[
F_r^d = \sigma^{r-1}(F_1^d) = \bigl\{\{\sigma^{r-1}(x),\sigma^{r-1}(y)\} \mid \{x,y\}\in F_1^d\bigr\}.
\]
Then each $F_r^d$ is necessarily a one-factor of $K_{2n}$ or $K_{2n+2}$, since it's obtained from the one-factor $F_1^d$ by an automorphism of the graph. This gives us a total of $2n-2$ one-factors for each graph, whereas a one-factorisation of $K_{2n}$ requires a total of $2n-1$, and a one-factorisation of $K_{2n+2}$ requires a total of $2n+1$. To construct a $(2n-1)$th one-factor for each graph we set
\begin{align*}
F_{2n-1}^1 &= \bigl\{ \{x,x+n-1\}  \mid  0 \le x \le n-2 \bigr\} \cup\bigl\{\{-\infty,\infty\}\bigr\}, \\
 &= \bigl\{\{0,n-1\},\{1,n\},\ldots,\{n-2,2n-3\},\{-\infty,\infty\}\bigr\}, \\
F_{2n-1}^2 &= F_{2n-1}^1\cup\bigl\{\{-i\infty,i\infty\}\bigr\}.
\end{align*}
These sets of edges are easily seen to meet each vertex of $K_{2n}$ and $K_{2n+2}$, respectively, exactly once. Moreover they have precisely $n$ edges in common, namely all $n$ edges of $F_{2n-1}^1$. 

\begin{figure} 
\begin{tikzpicture}[vertex/.style={circle,draw,fill=black!20,minimum size = 5mm,inner sep=0pt},>=stealth,thick]
\setlength{\radius}{3cm}
\foreach \x in {0,1,...,13}
   \node (\x-1) at (360*\x/14:\radius) [vertex] {$\scriptstyle\x$};
\node (-inf-1) at (180:0.4*\radius) [vertex] {$\scriptscriptstyle -\infty$};
\node (inf-1) at (0:0.4*\radius) [vertex] {$\scriptstyle\infty$};
\begin{scope}[xshift=2.5*\radius]
\foreach \x in {0,1,...,13}
   \node (\x-2) at (360*\x/14:\radius) [vertex] {$\scriptstyle\x$};
\node (-inf-2) at (180:0.35*\radius) [vertex] {$\scriptscriptstyle -\infty$};
\node (inf-2) at (0:0.35*\radius) [vertex] {$\scriptstyle\infty$};
\node (iinf) at (90:0.35*\radius) [vertex] {$\scriptstyle i\infty$};
\node (-iinf) at (270:0.35*\radius) [vertex] {$\scriptscriptstyle -i\infty$};
\end{scope}
\draw (0-1) -- (6-1);
\draw (1-1) -- (5-1);
\draw (2-1) -- (4-1);
\draw (7-1) -- (12-1);
\draw (8-1) -- (11-1);
\draw (9-1) -- (10-1);
\draw (3-1) -- (-inf-1);
\draw (13-1) -- (inf-1);
\draw (0-2) -- (6-2);
\draw (1-2) -- (5-2);
\draw (2-2) -- (4-2);
\draw (7-2) -- (12-2);
\draw (8-2) -- (11-2);
\draw (3-2) -- (-inf-2);
\draw (13-2) -- (inf-2);
\draw (9-2) -- (-iinf);
\draw (10-2) -- (iinf);
\end{tikzpicture}
\caption{The one-factors $F_1^1$ (left) and $F_1^2$ (right) in the case $n=8$.}
\label{fig:F1 for n=8}
\end{figure}
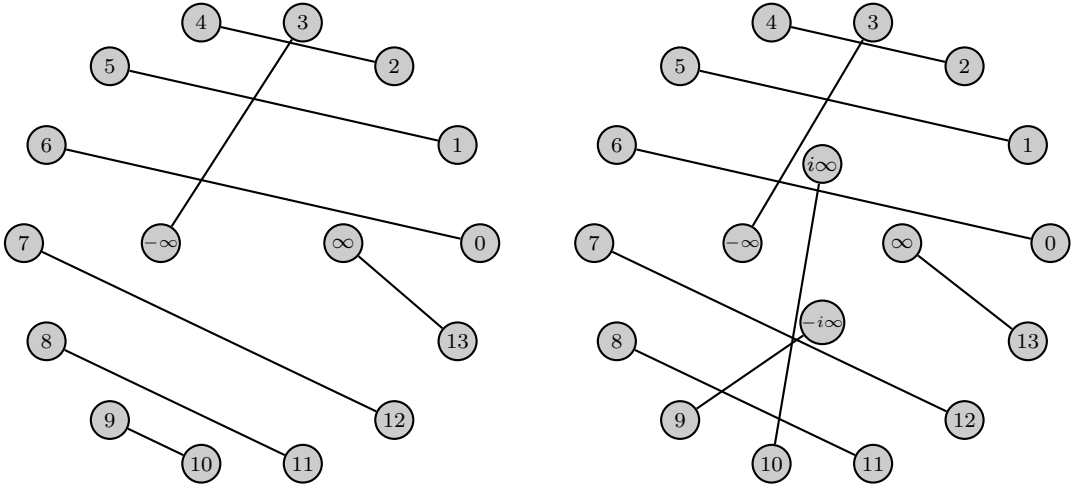

In order to construct the final two one-factors for $K_{2n+2}$ we must proceed carefully, in order to make sure we pick up an extra common fixture in each of rounds $2n$ and $2n+1$. The key point is to ensure that we place the edge $\{s,t\}$ or $\{u,v\}$ removed from $F_1^1$ when constructing $F_1^2$ in $F_{2n}^2$. This may be done as follows. When $n$ is even we set
\begin{align*}
T_1 &= \bigl\{\sigma^{2j}(\{u,v\}) \mid 0\leq j\leq n-2\bigl\}, \\
T_2 &= \sigma(T_1) \\
    &= \bigl\{\sigma^{2j+1}(\{u,v\}) \mid 0\leq j\leq n-2\bigl\},
\end{align*}
and when $n$ is odd we set
\begin{align*}
T_1 &= \bigl\{\sigma^{2j}(\{s,t\}) \mid 0\leq j\leq n-2\bigl\}, \\
T_2 &= \sigma(T_1) \\
    &= \bigl\{\sigma^{2j+1}(\{s,t\}) \mid 0\leq j\leq n-2\bigl\}.
\end{align*}
Since $v-u=1$ when $n$ is even, and $t-s=1$ when $n$ is odd, in all cases the sets $T_1$ and $T_2$ are the sets
\begin{align*}
T_{\mathrm{even}} &= \bigl\{\{2k,2k+1\} \mid 0\leq k\leq n-2\bigr\}, \\
T_{\mathrm{odd}} &= \bigl\{\{2k-1,2k\} \mid 0\leq k\leq n-2\bigr\}
\end{align*}
in some order. Just which is which depends on the value of $n$ modulo 4: 

\begin{enumerate}
\item For $n = 4\ell$, the vertex $u=\frac{3n-6}{2} = \frac{12\ell-6}{2} = 6\ell-3$ is odd, so $T_1=T_{\mathrm{odd}}$;
\item 
for $n = 4\ell + 1$, the vertex $s=\frac{n-3}{2} = \frac{4\ell-2}{2} = 2\ell-1$ is odd, so 
$T_1 = T_{\mathrm{odd}}$;
\item 
for $n = 4\ell+2$, the vertex  $u=\frac{3n-6}{2} = \frac{12\ell}{2} = 6\ell$ is even, so 
$T_1=T_{\mathrm{even}}$; and
\item 
for $n = 4\ell+3$, the vertex $s=\frac{n-3}{2} = \frac{4\ell}{2} = 2\ell$ is even, so 
$T_1=T_{\mathrm{even}}$.
\end{enumerate}
To complete $T_1$ and $T_2$ to one-factors of $K_{2n+2}$ we set
\begin{align*}
F_{2n}^2 &= T_1\cup\bigl\{\{-\infty,-i\infty\},\{\infty,i\infty\}\bigr\}, \\
F_{2n+1}^2 &= T_2\cup\bigl\{\{-\infty,i\infty\},\{\infty,-i\infty\}\bigr\}.
\end{align*}
Let
\begin{align*}
\mathcal{F}^1 & = \{F_r^1 \mid 1\leq r\leq 2n-1\}, &
\mathcal{F}^2 & = \{F_r^2 \mid 1\leq r\leq 2n+1\}. 
\end{align*}
We now claim:
\begin{theorem}
\label{lem:onefactorisation}
The set $\mathcal{F}^1$ is a one-factorisation of $K_{2n}$, and the set $\mathcal{F}^2$ is a one-factorisation of $K_{2n+2}$. Together these one-factorisations realise the upper bound of Corollary~\ref{cor:bound}.
\end{theorem}

\begin{proof}
We begin by understanding the orbits of the cyclic group $G=\langle\sigma\rangle$ of order $2n-2$ acting on the edges of $K_{2n}$ and $K_{2n+2}$.

For $\delta=1,\ldots,n-1$ let 
\[
O_\delta = \bigl\{\{x,x+\delta\} \mid 0\leq x\leq 2n-3\bigr\},
\]
where addition is carried out modulo $2n-2$, and for $\alpha\in\{\pm\infty,\pm i\infty\}$ let
\begin{align*}
O_\alpha = \bigl\{\{x,\alpha\} \mid 0\leq x\leq 2n-3\bigr\}.
\end{align*}
Finally, let also
\[
E^G = \bigl\{\{\alpha,\beta\} \mid \alpha,\beta\in\{\pm\infty,\pm i\infty\},\alpha\neq\beta\bigr\}.
\]
Then it is easily seen that each set $O_1,\ldots,O_{n-1},O_{-\infty},O_\infty,O_{-i\infty},O_{i\infty}$ is an orbit of $G$ acting on the edges of $K_{2n+2}$, and that $E^G$ is the fixed point set of this action. The orbits $O_\delta$ for $1\leq \delta\leq n-2$ and $O_\alpha$ for $\alpha\in\{\pm\infty,\pm i\infty\}$ have order $2n-2$, while the orbit $O_{n-1}$ has order $n-1$. This gives us a total of 
\begin{itemize}
\item
$n+2$ orbits of size $2n-2$ in $K_{2n+2}$, of which $n$ lie in $K_{2n}$;
\item
one orbit of size $n-1$ in $K_{2n+2}$, which also lies in $K_{2n}$;
\item
six orbits of size 1 in $K_{n+2}$, of which precisely one lies in $K_{2n}$.
\end{itemize} 
Together these account for all 
$(n+2)(2n-2)+(n-1)+6=2n^2+3n+1=(n+1)(2n+1)$ edges of $K_{2n+2}$, and all
$n(2n-2)+(n-1)+1=n(2n-1)$ edges of $K_{2n}$.

Beginning with $\mathcal{F}^1$, observe that in $E_1\subseteq F_1^1$ the differences between the vertices in each edge are
\begin{align*}
(n-2)-0  &= n-2,\\
(n-3)-1  &= n - 4,\\
         &\;\;\vdots\\
 t-s     &= \begin{cases}
            \frac{n}{2}-\frac{n-4}{2}=2,  & \text{$n$ even}\\
            \frac{n-1}{2}-\frac{n-3}{2}=1, &\text{$n$ odd},
            \end{cases}
 \end{align*}
while in $E_2\subseteq F_1^1$ the differences are given by
\begin{align*}
 (2n-4)-(n-1) &= n-3,\\
 (2n-5) - (n) &= n - 5,\\
 &\;\; \vdots\\
 v-u &= \begin{cases}
        \frac{3n-4}{2}-\frac{3n-6}{2}=1, & \text{$n$ even} \\
        \frac{3n-7}{2}-\frac{3n-3}{2}=2, & \text{$n$ odd}. 
        \end{cases}
\end{align*}
Together these differences are distinct and take all values from $1$ to $n-2$.
Consequently, $E_1\cup E_2$ contains precisely one edge from each orbit $O_1,\ldots,O_{n-2}$. In addition, the set $E_3$ contains precisely one edge from each orbit $O_{\pm\infty}$, and so in total $F_1^1$ contains precisely one representative from each of the $n$ orbits of size $2n-2$ lying in $K_{2n}$. Note also that $F_{2n-1}^1$ consists of $O_{n-1}$, together with the sole fixed edge $\{-\infty,\infty\}$ lying in $K_{2n}$. Since $F_r^1=\sigma^{r-1}(F_1^1)$ for $1\leq r\leq 2n-2$ it immediately follows that the $F_r^1$ are all disjoint, and together account for every edge of $K_{2n}$. The set $\mathcal{F}^1 = \{F_r^1 \mid 1\leq r\leq 2n-1\}$ therefore forms a one-factorisation of $K_{2n}$.

Turning now to $\mathcal{F}^2$, the one-factor $F_1^2$ is obtained from $F_1^1$ by deleting whichever edge $\{s,t\}$ or $\{u,v\}$ belongs to $O_1$, and replacing it with an edge from each of $O_{-i\infty}$ and $O_{i\infty}$. Consequently $F_1^2$ contains precisely one representative of $n+1$ of the orbits of size $2n-2$, namely $O_2,\ldots,O_{n-2}$ and $O_\alpha$ for $\alpha\in\{\pm\infty,\pm i\infty\}$. We again have $F_r^2=\sigma^{r-1}(F_1^2)$ for $1\leq r\leq 2n-2$, so it immediately follows that the $F_r^2$ are disjoint for $1\leq r\leq 2n-2$, with union $O_2\cup\cdots\cup O_{n-2}\cup O_{-\infty}\cup O_\infty\cup O_{-i\infty}\cup O_{i\infty}$. It's now easily checked the remaining one-factors $F_{2n-1}^2,F_{2n}^2,F_{2n+1}^2$ are disjoint with union $O_1\cup O_{n-1}\cup E^G$, and the claim that $\mathcal{F}^2$ is a one-factorisation of $K_{2n+2}$ follows.

We now count the common fixtures. By Lemma~\ref{lem:round1} we obtain $n-1$ common fixtures in round 1, and this gives us $n-1$ common fixtures in each round $r$ for  $1\leq r\leq2n-2$, since the draws for these rounds are obtained from those in round 1 by translation by $\sigma^{r-1}$. As observed above $|F_{2n-1}^1\cap F_{2n-1}^2|=n$, so we obtain $n$ common fixtures in round $2n-1$. By choice of $T_1$ we obtain a further common fixture in round $2n$, and since $T_2=\sigma(T_1)$, $F_2^1=\sigma(F_1^1)$, this gives another common fixture in round $2n+1$ also. Summing, we obtain the upper bound of Corollary~\ref{cor:bound}, as claimed.
\end{proof}

\begin{remark}
In the above construction only two of the common fixtures occur in the final two rounds of division two. Thus, if division one is played as a single round robin only, then our construction achieves a total of $c(n)-2$ common fixtures. Combined with the lower bound of Remark~\ref{rem:lowerbound}, this proves the claim of Remark~\ref{rem:singleroundrobin} that $c(n)-2$ is the maximum possible number of common fixtures in this case.
\end{remark}

\begin{remark}
Our general construction described in this section does not apply when $n=2$, because for $n=2$ the only orbits of order $2n-2=2$ are $O_{\pm\infty}$. In particular, for $n=2$ the orbit $O_1$ has order $n-1=1$ rather than 2.
\end{remark}

\subsection {Home and away status}
\label{sec:homeaway}

To complete the proof of Theorem~\ref{thm:main} it remains to show that the draws for all three round robins can be chosen to be balanced for $n\geq 3$, subject to the condition that the same club be designated the home team in both divisions in any common fixture. This amounts to orienting the edges of $K_{2n}$ and $K_{2n+2}$ in such a way that the indegree of each vertex differs from its outdegree by exactly one, and any edge corresponding to a common fixture is identically oriented in both graphs.

To achieve this we orient the edges of $K_{2n+2}$ belonging to each orbit of the action of $G$ as follows:
\begin{align*}
O_\delta &= \{(x,x+\delta) \mid 0\leq x\leq 2n-3\} & &\text{for $1\leq \delta\leq n-2$}, \\
O_{n-1} &= \{(x,x+(n-1)) \mid 0\leq x\leq n-2\},  \\
O_\alpha &= \bigcup_{k=0}^{n-2}\{(2k,\alpha),(\alpha,2k+1)\} &
                & \text{for $\alpha=\infty,i\infty$}, \\
O_\alpha &= \bigcup_{k=0}^{n-2}\{(2k+1,\alpha),(\alpha,2k)\} &
                & \text{for $\alpha=-\infty,-i\infty$},
\end{align*}
and
\[
E^G = \{(-\infty,\infty),(-\infty,-i\infty),(\infty,i\infty),(-i\infty,\infty),(-i\infty,i\infty),(i\infty,-\infty)\}.
\]
For $1\leq \delta\leq n-2$ each orbit $O_\delta$ is a disjoint union of cycles of length at least 3. Using this fact it is easily checked that orienting the edges as above achieves balance for the draw in division two, with the vertices belonging to $\{0,1,\ldots,n-2,-\infty,-i\infty\}$ having indegree $n$ and outdegree $n+1$, and the vertices belonging to $\{n-1,n,\ldots,2n-3,\infty,i\infty\}$ having indegree $n+1$ and outdegree $n$. 

As a first step towards achieving balance in division one we regard $K_{2n}$ as a subgraph of $K_{2n+2}$, and give each edge of $K_{2n}$ the orientation it receives as an edge of $K_{2n+2}$. The resulting draw is balanced, and edges corresponding to common fixtures in rounds $1$ to $2n-1$ are identically oriented. However, the two edges corresponding to common fixtures in rounds $2n$ and $2n+1$ are oppositely oriented, because the orientations of the edges of $K_{2n}$ are reversed in rounds $2n$ to $4n-2$. But this is easily remedied, because these edges both belong to $O_1$, and are the only edges in this orbit that occur in common fixtures. Thus we may achieve our goal by simply reversing the orientation in $K_{2n}$ of all edges belonging to $O_1$, which has no effect on the balance. This completes the proof of Theorem~\ref{thm:main}.

\begin{remark}
When division one is played as a single round robin only, the final step of reversing the orientation of $O_1$ is unnecessary, and we may achieve balance in both divisions one and two by simply orienting $K_{2n}$ as a subgraph of $K_{2n+2}$.
\end{remark}


\begin{thebibliography}{10}

\bibitem{bib:anderson1997}
I.~Anderson.
\newblock Balancing carry-over effects in tournaments.
\newblock In {\em Combinatorial designs and their applications ({M}ilton
  {K}eynes, 1997)}, volume 403 of {\em Chapman \& Hall/CRC Res. Notes Math.},
  pages 1--16. Chapman \& Hall/CRC, Boca Raton, FL, 1999.

\bibitem{bib:Beecham&Hurley}
A.~F. Beecham and A.~C. Hurley.
\newblock A scheduling problem with a simple graphical solution.
\newblock {\em J. Austral. Math. Soc. Ser. B}, 21(4):486--495, 1979/80.

\bibitem{bib:briskorn2009}
D.~Briskorn.
\newblock Combinatorial properties of strength groups in round robin
  tournaments.
\newblock {\em European J. Oper. Res.}, 192(3):744--754, 2009.

\bibitem{bib:briskorn-kunst2010}
D.~Briskorn and S.~Knust.
\newblock Constructing fair sports league schedules with regard to strength
  groups.
\newblock {\em Discrete Appl. Math.}, 158(2):123--135, 2010.

\bibitem{bib:dewerra1980}
D.~de~Werra.
\newblock Geography, games and graphs.
\newblock {\em Discrete Appl. Math.}, 2(4):327--337, 1980.

\bibitem{bib:deWerra&Jacot-Descombes&Masson}
D.~de~Werra, L.~Jacot-Descombes, and P.~Masson.
\newblock A constrained sports scheduling problem.
\newblock {\em Discrete Appl. Math.}, 26(1):41--49, 1990.

\bibitem{bib:DellaCroce&Tadei&Asioli}
F.~Della~Croce, R.~Tadei, and P.~S. Asioli.
\newblock Scheduling a round robin tennis tournament under courts and players
  availability constraints.
\newblock {\em Ann. Oper. Res.}, 92:349--361, 1999.
\newblock Models and algorithms for planning and scheduling problems: Cambridge
  Workshop (1997).

\bibitem{bib:Froncek}
D.~Fron{\v{c}}ek.
\newblock Scheduling the {C}zech {N}ational {B}asketball {L}eague.
\newblock In {\em Proceedings of the {T}hirty-second {S}outheastern
  {I}nternational {C}onference on {C}ombinatorics, {G}raph {T}heory and
  {C}omputing ({B}aton {R}ouge, {LA}, 2001)}, volume 153, pages 5--24, 2001.

\bibitem{bib:Gelling&Odeh}
E.~N. Gelling and R.~E. Odeh.
\newblock On {$1$}-factorizations of the complete graph and the relationship to
  round robin schedules.
\newblock In {\em Proceedings of the {T}hird {M}anitoba {C}onference on
  {N}umerical {M}athematics ({W}innipeg, {M}an., 1973)}, pages 213--221,
  Winnipeg, Man., 1974. Utilitas Math.

\bibitem{bib:Hoshino}
R.~Hoshino and K.~Kawarabayashi.
\newblock A multi-round generalization of the traveling tournament problem and
  its application to {J}apanese baseball.
\newblock {\em European J. Oper. Res.}, 215(2):481--497, 2011.

\bibitem{bib:Kendall}
G.~Kendall, S.~Knust, C.~C. Ribeiro, and S.~Urrutia.
\newblock Scheduling in sports: An annotated bibliography.
\newblock {\em Computers and Operations Research}, 37(1):1--19, 2010.

\bibitem{bib:knust2008}
S.~Knust.
\newblock Scheduling sports tournaments on a single court minimizing waiting
  times.
\newblock {\em Oper. Res. Lett.}, 36(4):471--476, 2008.

\bibitem{bib:Mazzuoccolo&Rinaldi}
G.~Mazzuoccolo and G.~Rinaldi.
\newblock {$k$}-pyramidal one-factorizations.
\newblock {\em Graphs Combin.}, 23(3):315--326, 2007.

\bibitem{bib:Mendelsohn&Rosa}
E.~Mendelsohn and A.~Rosa.
\newblock One-factorizations of the complete graph --- a survey.
\newblock {\em Journal of Graph Theory}, 9(1):43--65, 1985.

\bibitem{bib:Hof}
P.~van~'t Hof, G.~Post, and D.~Briskorn.
\newblock Constructing fair round robin tournaments with a minimum number of
  breaks.
\newblock {\em Oper. Res. Lett.}, 38(6):592--596, 2010.

\bibitem{bib:Wallis}
W.~D. Wallis.
\newblock {\em One-factorizations}, volume 390 of {\em Mathematics and its
  Applications}.
\newblock Kluwer Academic Publishers Group, Dordrecht, 1997.

\end{thebibliography}
\end{document}